\documentclass[11pt]{article}
\usepackage{amsmath, amsthm, amstext}
\usepackage{amssymb, array, amsfonts}
\usepackage[utf8]{inputenc}
\usepackage[english]{babel}
\usepackage{enumerate}
\usepackage{graphicx}
\usepackage{mathrsfs}

\usepackage{pdfpages}

\inputencoding{utf8}

\usepackage{pagecolor}
\definecolor{back}{RGB}{240, 238, 220}

\newtheorem{thm}{Theorem}
\newtheorem{lemma}{Lemma}

\title{Existence of a solution to the scattering problem for the ultrahyperbolic equation}
\author{M. N. Demchenko\footnote{St.~Petersburg Department of
V.\,A.~Steklov Institute of Mathematics of
the Russian Academy of Sciences, 
27 Fontanka, St.~Petersburg, Russia. E-mail: demchenko@pdmi.ras.ru}}
\date{}

\begin{document}
\maketitle
\begin{flushright}
{\em\large
\begin{tabular}{r}
To the memory of Natalia Yakovlevna Kirpichnikova
\end{tabular}
}
\end{flushright}

\begin{abstract}
We consider the ultrahyperbolic equation in the Euclidean space. 
The behavior at the infinity of a certain class of solutions is studied.
We examine the issue of existence of solutions to the scattering problem:
for a given asymptotics at the infinity the corresponding solution to the equation is constructed.

\medskip


\end{abstract}

\section{Introduction}
We study the ultrahyperbolic equation of the form
\begin{gather}
  (\Delta_y - \Delta_x) u = 0,
  \\
  \Delta_x = \partial_{x_1}^2 + \ldots + \partial_{x_d}^2, \quad \Delta_y = \partial_{y_1}^2 + \ldots + \partial_{y_n}^2, \notag
  \label{eqn}
\end{gather}
where a solution $u(x,y)$ is a function in
${\mathbb R}^d\times{\mathbb R}^n$, $d,n \geqslant 1$.
In the particular cases when either $d=1$ or $n=1$, this equation becomes the (hyperbolic) wave equation.

We shall consider $C^\infty$-smooth solutions to equation~(\ref{eqn}) exhibiting the following asymptotic behavior at the infinity
\begin{equation}
  u(s \theta, (s + p) \omega) = s^{-N/2+1} f(\theta,\omega,p)
  + O(s^{-N/2+1-\varepsilon}), \quad s\to +\infty,    
  \label{scat-data}
\end{equation}
with some $0<\varepsilon\leqslant 1/2$, $(\theta,\omega,p)\in S^{d-1}\times S^{n-1}\times{\mathbb R}$.
Here $N=d+n$, $S^{d-1}$ ($S^{n-1}$) is the unit sphere in ${\mathbb R}^d$ (${\mathbb R}^n$) centered at the origin.
The function $f(\theta,\omega,p)$ will be called the scattering data of a solution $u(x,y)$.
We shall find a class of solutions exhibiting the specified asymptotics, and, besides,
we shall construct a solution to equation~(\ref{eqn}) satisfying~(\ref{scat-data}) for a given function $f(\theta,\omega,p)$.
The latter means that the existence of a solution to the scattering problem for equation~(\ref{eqn}) will be established.

In order to formulate the regularity conditions on the scattering data,
we first clarify what we mean by the derivatives of functions on the unit sphere.
We shall consider continuations of such functions to the entire Euclidean space
homogeneous of degree $0$ (these will be denoted by the same symbols). 
By the corresponding partial derivatives we shall mean the ones of such continuations taken on the unit sphere.

We shall also use the Hilbert transform $H$ of a function on ${\mathbb R}$:
\begin{equation*}
  (H f)(p) = \frac{1}{\pi}\, \mathrm{v.p.} \int_{\mathbb R} \frac{f(p')}{p-p'} \, dp'.
\end{equation*}
In terms of the inverse Fourier transform (the definition of the Fourier transform is given in sec.~\ref{sec-family}),
this definition can be reformulated as follows
\begin{equation}
  (H f)\check{}\,(r) = (i\, {\rm sgn}\, r) \check f(r). 
  \label{hilbert-check}
\end{equation}

Now we can formulate the theorem of existence of a solution to the scattering problem for equation~(\ref{eqn}).

\begin{thm}\label{thm-exist}
Let a function $f(\theta, \omega, p)$ on $S^{d-1}\times S^{n-1}\times{\mathbb R}$ satisfy the following conditions
\begin{gather}
  \partial^\alpha_{\theta,\omega}f(\theta, \omega, \pm\infty) = 0, \label{f-reg1}
  \\
  |\partial_p^k\partial^\alpha_{\theta,\omega} f(\theta, \omega, p)| \leqslant C_{k,\alpha} (1+|p|)^{-k-\varepsilon} 
  \label{f-reg2}
\end{gather}
for all $k\geqslant 1$, $|\alpha|\geqslant 0$, and some $0<\varepsilon\leqslant 1/2$,
and let the following relation hold:
\begin{equation}
  f(-\theta, -\omega, p) = (H^{d-n} f(\theta, \omega, \cdot))(-p).
  \label{f-hilbert}
\end{equation}
Then there exists a $C^\infty$-smooth solution $u(x,y)$ to equation~(\ref{eqn}) satisfying~(\ref{scat-data}).
\end{thm}

Note that the Hilbert transform in condition~(\ref{f-hilbert}) is well-defined for functions sa\-tis\-fy\-ing~(\ref{f-reg1}), (\ref{f-reg2}).
This follows from Lemma~\ref{lemma-V-eps-high-order-final}, according to which the inverse Fourier transform of $f(\theta, \omega, \cdot)$
is a regular function in ${\mathbb R}$, and so the Hilbert transform may be defined by relation~(\ref{hilbert-check}).

Condition~(\ref{f-hilbert}) on the scattering data seems to be necessary for the existence of a solution from a reasonable class of functions.
In the case $d=3$, $n=1$ (which corresponds to the wave equation in the $3$-dimensional space), it takes the following form
\begin{equation*}
  f(-\theta, -\omega, p) = -f(\theta, \omega, -p)
\end{equation*}
(in this case, $\omega=\pm 1$),
and is automatically satisfied for all solutions to the Cauchy problem with smooth rapidly decaying data given
at some initial moment~\cite{Blag-scatter}.
Note also that in the case of arbitrary $d$, $n$,
we shall find a class of solutions exhibiting asymptotics~(\ref{scat-data}),
for which relation~(\ref{f-hilbert}) is also automatically satisfied (Theorem~\ref{thm-U-class}).

The problem for equation~(\ref{eqn}) connected with the scattering data in the hyperbolic case, when either $d=1$ or $n=1$,
was studied in a number of papers.
We mention the pioneering ones on this subject~\cite{Lax-Phillips, Blag-scatter}, and also papers~\cite{Moses, Kis, Plachenov, BV},
in which the asymptotic behavior of solutions to hyperbolic equations at the infinity was investigated.
The general case $d,n\geqslant 1$ is studied to much less extent. 
Note that the set of known well-posed problems for ultrahyperbolic equations 
is considerably less than
that for elliptic, parabolic, and hyperbolic equations.
Paper~\cite{Blag-char} deals with the characteristic problem for equation~(\ref{eqn}),
in which a solution $u(x,y)$ is to be determined in the region $|x|<|y|$ (or $|x|>|y|$) from its values
on the characteristic cone $|x|=|y|$.
The result of the paper just mentioned is the existence, uniqueness, and a formula for a solution to the problem in a certain class.
The paper~\cite{DMN} deals with 
the scattering problem for the ultrahyperbolic equation, which differs from~(\ref{eqn}) by a constant nonzero potential
(in the hyperbolic case, it becomes Klein-Gordon-Fock equation).

\section{A family of solutions exhibiting asymptotics~(\ref{scat-data})}  
\label{sec-family}
We use the following definition of the Fourier transform
\begin{gather*}
  \hat v(\xi) = \int_{{\mathbb R}^d} e^{-i\langle{}x,\xi\rangle} v(x)  dx,
  \quad
  v(x) = (2\pi)^{-d} \int_{{\mathbb R}^d} e^{i\langle{}x,\xi\rangle} \hat v(\xi)  d\xi
\end{gather*}
(here and further $\langle\cdot,\cdot\rangle$ denotes the standard inner product in the real Euclidean space;
angle brackets will also denote the pairing of a distribution with a test function).
For a function $v(x,y)$ in the ``spacetime'' ${\mathbb R}^d\times{\mathbb R}^n$,
it is convenient to define the Fourier transform as follows
\begin{gather*}
  \hat v(\xi,\eta) = \int_{{\mathbb R}^d\times{\mathbb R}^n} e^{i(-\langle{}x,\xi\rangle+\langle{}y, \eta\rangle)} v(x,y)  dx dy.
\end{gather*}
Then the inverse Fourier transform is given by the formula 
\begin{gather*}
    v(x,y) = (2\pi)^{-N} \int_{{\mathbb R}^d\times{\mathbb R}^n} e^{i(\langle{}x,\xi\rangle-\langle{}y, \eta\rangle)} \hat v(\xi,\eta) \, d\xi d\eta.
\end{gather*}

The Fourier transform $\hat u(\xi,\eta)$ of a solution $u(x,y)$ is supported on the cone $\{|\xi|=|\eta|\}$.
Formally, this can be written as the following relation
\begin{equation}
  \langle\hat u, \varphi\rangle = \int_{S^{d-1}} d\zeta \int_{S^{n-1}} d\sigma \int_0^\infty A(\zeta,\sigma,r) \varphi(r\zeta,r\sigma) \, dr
  \label{hatu-f}
\end{equation}
(here and further in similar situations, $d\zeta$ ($d\sigma$) denotes the surface measure on the sphere $S^{d-1}$ ($S^{n-1}$)),
where $\varphi(\xi,\eta)$ is an arbitrary test function from the Schwartz class $\mathcal{S}({\mathbb R}^d\times{\mathbb R}^n)$
with some function $A(\zeta,\sigma,r)$ in $S^{d-1}\times S^{n-1}\times {\mathbb R}_+$.

We shall consider functions $A(\zeta,\sigma,r)$ satisfying the following regularity condition for some $0<\varepsilon\leqslant 1/2$:
\begin{equation}
  |\partial^k_r\partial_{\zeta,\sigma}^\alpha A(\zeta,\sigma,r)| \leqslant C_{\ell,k,\alpha} r^{N/2-k-2+\varepsilon} (1+r)^{-\ell}, \quad \ell,k,|\alpha|\geqslant 0.
  \label{U-cond-pointwise}
\end{equation}
This condition provides that the right hand side of~(\ref{hatu-f}) is well-defined
and corresponds to a tempered distribution from $\mathcal{S}'({\mathbb R}^d\times{\mathbb R}^n)$.
The latter is a solution to equation~(\ref{eqn}).
Indeed, for a test function $\psi(x,y)$, we have
\begin{equation*}
  ((\Delta_y - \Delta_x) \psi)\check{}\,(\xi,\eta) = (\xi^2-\eta^2) \check\psi(\xi,\eta),
\end{equation*}
whence
\begin{equation*}
  \langle{}u, (\Delta_x - \Delta_y) \psi\rangle = \langle\hat u, (\xi^2-\eta^2) \check\psi\rangle.
\end{equation*}
The resulting expression is equal to zero, in view of~(\ref{hatu-f}).

\begin{thm}\label{thm-U-class}
  Let a function $A(\zeta,\sigma,r)$ satisfy condition~(\ref{U-cond-pointwise}) for some $0<\varepsilon\leqslant 1/2$.
  Then the function $u(x,y)$ defined by relation~(\ref{hatu-f})
  is a $C^\infty$-smooth solution to equation~(\ref{eqn}), and relation~(\ref{scat-data}) holds with
\begin{equation}
  f(\theta, \omega, p) = c \int_{\mathbb R} e^{-i r p} 
  \left(e^{i\pi (n-d)/4} \Theta(r) + e^{i\pi (d-n)/4} \Theta(-r)\right) |r|^{-N/2+1} A(\theta, \omega, r) \, dr,
  \label{fUR}
\end{equation}
where $c = (2\pi)^{-N/2-1}$,
and the function $A$ is continued to $r\in{\mathbb R}$ by the rule
\begin{equation}
  A(\zeta, \sigma, -r) = A(-\zeta, -\sigma, r). 
  \label{U-line}
\end{equation}
Furthermore, 
$f(\theta, \omega, p)$ is a $C^\infty$-smooth function on $S^{d-1}\times S^{n-1}\times{\mathbb R}$ satisfying
\begin{equation}
  f(\theta,\omega,\cdot) \in L_q({\mathbb R}), \quad (\theta,\omega) \in S^{d-1}\times S^{n-1}
  \label{fLq}
\end{equation}
for any $1/\varepsilon < q \leqslant \infty$,
and relation~(\ref{f-hilbert}) holds.
\end{thm}

Note that in view of inclusion~(\ref{fLq}), the Hilbert transform of the function $f(\theta,\omega,\cdot)$
occurring in relation~(\ref{f-hilbert}) is well-defined.

We shall derive asymptotics~(\ref{scat-data}) and formula~(\ref{fUR}) for solutions $u(x,y)$ of the form~(\ref{hatu-f})
with the use of the method of stationary phase.
To perform this, it is convenient to use the more explicit representation for such solutions.
According to~(\ref{hatu-f}), for a test function $\psi(x,y)$, we have
\begin{multline*}
  \langle{}u, \psi\rangle = \langle\hat u, \check\psi\rangle
  \\= \frac{1}{(2\pi)^N} \int_0^\infty dr \int_{S^{d-1}\times S^{n-1}} d\zeta d\sigma
  A(\zeta, \sigma, r) \int_{{\mathbb R}^d\times{\mathbb R}^n} e^{i r(\langle{}x, \zeta\rangle - \langle{}y, \sigma\rangle)} \psi(x,y) dx dy.
\end{multline*}
In view of condition~(\ref{U-cond-pointwise}), we obtain an absolutely convergent integral, which equals
\begin{equation*}
  \frac{1}{(2\pi)^N} \int_{{\mathbb R}^d\times{\mathbb R}^n} dx dy\, \psi(x,y) \int_0^\infty dr \int_{S^{d-1}\times S^{n-1}} 
  e^{i r(\langle{}x, \zeta\rangle - \langle{}y, \sigma\rangle)} A(\zeta, \sigma, r) d\zeta d\sigma,
\end{equation*}
whence
\begin{equation}
  u(x,y) = \frac{1}{(2\pi)^N} \int_0^\infty dr \int_{S^{d-1}\times S^{n-1}} 
  e^{i r(\langle{}x, \zeta\rangle - \langle{}y, \sigma\rangle)} A(\zeta, \sigma, r) d\zeta d\sigma.
  \label{u-simple}
\end{equation}
The last relation implies that $u(x,y)$ is a $C^\infty$-smooth function in ${\mathbb R}^d\times{\mathbb R}^n$,
since the function $A$ decays rapidly for large $r$, due to condition~(\ref{U-cond-pointwise}).

The derivation of asymptotics~(\ref{scat-data}) and formula~(\ref{fUR}) will be given in sec.~\ref{sec-asymp}.
In the rest of this section, we shall prove the remaining assertions of Theorem~\ref{thm-U-class},
and prove Theorem~\ref{thm-exist}.

It follows from formula~(\ref{fUR}) that the function $f(\theta,\omega,p)$ is a Fourier transform of the function
\begin{equation}
  \check f(\theta, \omega, r) = c \left(e^{i\pi (n-d)/4} \Theta(r) + e^{i\pi (d-n)/4} \Theta(-r)\right) |r|^{-N/2+1} A(\theta, \omega, r)
  \label{fU}
\end{equation}
with respect to $r$, allowing the following estimate
\begin{equation*}
  |\check f(\theta, \omega, r)| \lesssim |r|^{-1+\varepsilon} (1+|r|)^{-\ell}.
\end{equation*}
Therefore $\check f(\theta, \omega, r)$ is an integrable function, which rapidly decays as $r\to\infty$,
and so $f(\theta,\omega,p)$ is a $C^\infty$-smooth function in $p$.
Besides, it is clear that $\check f(\theta, \omega, \cdot) \in L_q({\mathbb R})$, $1\leqslant q < 1/(1-\varepsilon)$,
and so by Hausdorff-Young theorem~\cite[Theorem~7.1.13]{Hormander}, inclusion~(\ref{fLq}) holds.

Now we are able to infer relation~(\ref{f-hilbert}) from equality~(\ref{fU}).
The latter, together with the rule~(\ref{U-line}), yields
\begin{equation*}
  \check f(-\theta, -\omega, r) = c \left(e^{i\pi (n-d)/4} \Theta(r) + e^{i\pi (d-n)/4} \Theta(-r)\right) |r|^{-N/2+1} A(\theta, \omega, -r).
\end{equation*}
By applying again~(\ref{fU}), we obtain the equality
\begin{equation*}
  \check f(-\theta, -\omega, r) 
  =\check f(\theta, \omega, -r)\, \frac{e^{i\pi (n-d)/4} \Theta(r) + e^{i\pi (d-n)/4} \Theta(-r)}{e^{i\pi (d-n)/4} \Theta(r) + e^{i\pi (n-d)/4} \Theta(-r)},
\end{equation*}
which can be written as follows
\begin{equation}
  \check f(-\theta, -\omega, r) = \check f(\theta, \omega, -r) (-i\,{\rm sgn}\, r)^{d-n}.
  \label{f-hilbert-fourier}
\end{equation}
Applying the inverse Fourier transform with respect to $r$ leads to relation~(\ref{f-hilbert}).

Now turn to the proof of Theorem~\ref{thm-exist}.
For $r>0$, we set
\begin{equation}
  A(\zeta, \sigma, r) = \check f(\zeta, \sigma, r) c^{-1} e^{i\pi (d-n)/4} r^{N/2-1},
  \label{Af}
\end{equation}
where $\check f(\zeta, \sigma, r)$, as in the previous argument, is the inverse Fourier transform of the function $f(\zeta, \sigma, p)$
with respect to $p$.
According to Lemma~\ref{lemma-V-eps-high-order-final}, conditions~(\ref{f-reg1}), (\ref{f-reg2}) on the function $f$ imply
that $\check f(\theta, \omega, \cdot)$ is a regular function in ${\mathbb R}$.
Moreover, estimate~(\ref{V-est-final}) means that the function $A$ given by~(\ref{Af}) satisfies condition~(\ref{U-cond-pointwise}) with $\alpha=0$.
The last condition with arbitrary $\alpha$ can be verified analogously.

Now Theorem~\ref{thm-U-class} allows us to define the function $u(x,y)$ by relation~(\ref{hatu-f})
and thus obtain a smooth solution to equation~(\ref{eqn}).
Besides, this solution obeys~(\ref{scat-data}) with $f$ replaced by the function $g$ related to $A$ by~(\ref{fUR}).
It remains to substitute the function $A$ of the form~(\ref{Af}) in this relation and to infer that $g=f$.
This can be done by a straightforward calculation with the use of condition~(\ref{f-hilbert}) on the function $f$,
or, more precisely, its equivalent form~(\ref{f-hilbert-fourier}).

\section{Asymptotic behavior of solutions of the form~(\ref{hatu-f})}\label{sec-asymp}
In view of equality~(\ref{u-simple}), we have
\begin{equation}
  u(s\theta, (s+p)\omega) =
  \frac{1}{(2\pi)^N} \int_0^\infty dr \int_{S^{d-1}\times S^{n-1}} e^{i r s \left(\langle\theta,\zeta\rangle - \langle\omega, \sigma\rangle\right)}
  e^{-i r p \langle\omega,\sigma\rangle} A(\zeta, \sigma, r) \, d\zeta d\sigma.
  \label{rSS}
\end{equation}

We shall find the asymptotics of the inner integral in~(\ref{rSS}) as $s\to +\infty$ by means of the method of stationary phase.
The function
\begin{equation*}
  Q(\zeta,\sigma) = \langle\theta,\zeta\rangle - \langle\omega, \sigma\rangle
\end{equation*}
occurring in the first exponential has four critical points:
\begin{equation}
  (\zeta,\sigma) \in \{ (\theta,\omega), (-\theta,-\omega), (-\theta,\omega), (\theta,-\omega)\}.
  \label{crit-pts}
\end{equation}
To calculate the contribution to the asymptotics of a point from this set, we introduce local coordinates
$(\zeta',\sigma')\in{\mathbb R}^{d-1}\times {\mathbb R}^{n-1}$ on the manifold $S^{d-1}\times S^{n-1}$ in a neighborhood of a given point.
We assume that a critical point is mapped to $(\zeta',\sigma') = (0,0)$.
Consider the integral
\begin{equation}
  \int_{{\mathbb R}^{d-1}\times {\mathbb R}^{n-1}} e^{i r s Q(\zeta,\sigma)}
  e^{-i r p \langle\omega,\sigma\rangle} A(\zeta, \sigma, r) J(\zeta',\sigma')\, d\zeta' d\sigma',
  \label{int-local}
\end{equation}
where $J(\zeta',\sigma')$ is some smooth compactly supported function that, in a neighborhood of the origin, 
is equal to the absolute value of the determinant of the Jacobian matrix
corresponding to chosen local coordinates.
Suppose also that the preimage of the support of this function on $S^{d-1}\times S^{n-1}$
contains exactly one critical point of the function $Q$.
For the first two critical points $(\zeta,\sigma) = (\pm\theta,\pm\omega)$ in~(\ref{crit-pts}),
the leading term of the asymptotics of the integral in question has the form
\begin{equation*}
  \left(\frac{2\pi}{r s}\right)^{N/2-1} |\det\partial^2_{\zeta',\sigma'} Q|^{-1/2}
  e^{i\pi\, {\rm sgn}\left(\partial^2_{\zeta',\sigma'}Q\right)/4} e^{\mp i r p} A J\big|_{(\zeta',\sigma')=(0,0)}.
\end{equation*}
Without loss of generality, we may assume that
$\theta = e_d$, $\omega = e_n$ ($e_j$ are elements of the standard basis in the Euclidean space).
Consider the critical point $(\zeta,\sigma) = (\theta,\omega)$.
In this case, a convenient choice of local coordinates is the following
\begin{multline*}
  (\zeta',\sigma') 
   \mapsto (\zeta,\sigma) = \\ =
  \left(\zeta_1', \ldots, \zeta_{d-1}', \sqrt{1-\zeta_1'^2-\ldots-\zeta_{d-1}'^2}, \sigma_1', \ldots, \sigma_{n-1}', \sqrt{1-\sigma_1'^2-\ldots-\sigma_{n-1}'^2}\right).
\end{multline*}
In this coordinates, we have
\begin{gather*}
  J(0,0) = 1,\\
  Q = \zeta_d - \sigma_n,
  \quad
  \partial_{\zeta_k'}Q = \frac{-\zeta_k'}{\zeta_d},
  \quad
  \partial_{\sigma_k'}Q = \frac{\sigma_k'}{\sigma_n},
  \\
  \partial^2_{\zeta_j' \zeta_k'}Q = \frac{-\delta_{jk}}{\zeta_d}
  - \frac{\zeta_j' \zeta_k'}{\zeta_d^3}.
  \quad
  \partial^2_{\sigma_j' \sigma_k'}Q =  \frac{\delta_{jk}}{\sigma_n}
  + \frac{\sigma_j' \sigma_k'}{\sigma_n^3},
  \quad
  \partial^2_{\zeta_j' \sigma_k'}Q = 0.
\end{gather*}
Hence for $(\zeta',\sigma')=(0,0)$, we have the following
\begin{equation*}
  |\det\partial^2_{\zeta',\sigma'}Q| = 1,
  \quad
  {\rm sgn}\left(\partial^2_{\zeta',\sigma'}Q\right) = n-d.
\end{equation*}
Thus the integral~(\ref{int-local}) has the following asymptotics as $s\to+\infty$:
\begin{equation*}
  \left(\frac{2\pi}{r s}\right)^{N/2-1} 
  e^{i\pi (n-d)/4} e^{-i r p} A(\theta, \omega, r).
\end{equation*}
The remainder term here 
is dominated by~\cite[Theorem~7.7.5]{Hormander} 
\begin{equation*}
  \frac{C(A, r)}{(r s)^{N/2-1/2}},
\end{equation*}
where $C(A,r)$ is bounded (up to a constant multiple) by the $C^N({\mathbb R}^{d-1}\times{\mathbb R}^{n-1})$-norm of the product
\begin{equation*}
  e^{-i r p \langle\omega,\sigma\rangle} A(\zeta, \sigma, r) J(\zeta',\sigma')
\end{equation*}
considered as a function in $\zeta'$, $\sigma'$ 
(in the case of odd $N$, one can take the norm in $C^{N-1}({\mathbb R}^{d-1}\times{\mathbb R}^{n-1})$).
We have
\begin{equation*}
  C(A,r) \lesssim (1 + r^N)\, \|A(\cdot, \cdot, r)\|_{C^N(S^{d-1}\times S^{n-1})} 
\end{equation*}
(the dependence of $C(A,r)$ on the parameter $p$ is not indicated explicitly, since the latter
is assumed to be fixed).
This and condition~(\ref{U-cond-pointwise}) with $k=0$, $|\alpha|\leqslant N$, give
\begin{equation*}
  C(A,r) \lesssim r^{N/2-2+\varepsilon} (1+r)^{-\ell}.
\end{equation*}
Thus we have the following estimate for the remainder term
\begin{equation}
  C s^{-N/2+1/2} r^{-3/2+\varepsilon} (1+r)^{-\ell}.
  \label{remainder}
\end{equation}

The same estimate for the remainder term holds for other critical points in~(\ref{crit-pts}).
By the standard fashion (by means of a partition of unity on $S^{d-1}\times S^{n-1}$),
the difference of the inner integral in~(\ref{rSS}) and the sum of integrals of the form~(\ref{int-local}) for all critical points
can also be estimated by~(\ref{remainder}).

In the consideration of the critical point $(\zeta,\sigma) = (-\theta,-\omega)$,
the sign of the square roots in the definition of local coordinates should be changed.
This leads to the equalities
\begin{equation*}
  |\det\partial^2_{\zeta',\sigma'}Q| = 1,
  \quad
  {\rm sgn}\left(\partial^2_{\zeta',\sigma'}Q\right) = d-n.
\end{equation*}
Then the corresponding contribution to the asymptotics takes the following form
\begin{equation*}
  \left(\frac{2\pi}{r s}\right)^{N/2-1} 
  e^{i\pi (d-n)/4} e^{i r p} A(-\theta, -\omega, r).
\end{equation*}
We give the contributions of the rest two critical points $(\zeta,\sigma) = (\mp\theta,\pm\omega)$
without specifying constants $C_1$, $C_2$:
\begin{equation*}
  \left(\frac{2\pi}{r s}\right)^{N/2-1}
  \left(C_1 e^{-2 i r s} e^{-i r p} A(-\theta, \omega, r)
    + C_2 e^{2 i r s} e^{i r p} A(\theta, -\omega, r)\right).
\end{equation*}

Thus the inner integral in~(\ref{rSS}) has the following asymptotics as $s\to+\infty$: 
\begin{multline}
  \left(\frac{2\pi}{r s}\right)^{N/2-1} 
  e^{i\pi (n-d)/4} \left\{e^{-i r p} A(\theta, \omega, r) + i^{d-n} e^{i r p} A(-\theta, -\omega, r) \right.
  \\ + \left. C_1 e^{-2 i r s} e^{-i r p} A(-\theta, \omega, r) + C_2 e^{2 i r s} e^{i r p} A(\theta, -\omega, r)\right\},
  \label{sphere-asymp}
\end{multline}
while the remainder term is bounded by~(\ref{remainder}).

Now turn to the integral with respect to $r$ in~(\ref{rSS}).
We decompose it to the sum of two integrals over intervals $0<r<1/s$ and $r>1/s$.
The integral of the majorant~(\ref{remainder}) for remainder term over $r>1/s$
is dominated by
\begin{equation*}
  C s^{-N/2+1/2} \int_{1/s}^\infty r^{-3/2+\varepsilon} (1+r)^{-\ell} dr \lesssim
  s^{-N/2 + 1 - \varepsilon}.
\end{equation*}
The integral corresponding to the first two terms in figure brackets in~(\ref{sphere-asymp}) over $r>1/s$
equals
\begin{equation*}
  \left(\frac{2\pi}{s}\right)^{N/2-1} e^{i\pi (n-d)/4} 
  \int_{1/s}^\infty r^{-N/2+1} 
  \left\{e^{-i r p} A(\theta, \omega, r) + i^{d-n} e^{i r p} A(-\theta, -\omega, r)\right\} dr.
\end{equation*}
Condition~(\ref{U-cond-pointwise}) with $k=|\alpha|=0$ and sufficiently large $\ell$ implies
that the integral in the last expression tends to the corresponding integral over $r>0$ as $s\to+\infty$.
Therefore this expression equals
\begin{multline*}
  \left(\frac{2\pi}{s}\right)^{N/2-1} e^{i\pi (n-d)/4} 
  \times\\\times \int_0^\infty r^{-N/2+1} 
  \left\{e^{-i r p} A(\theta, \omega, r) + i^{d-n} e^{i r p} A(-\theta, -\omega, r)\right\} dr
  + O(s^{-N/2+1-\varepsilon}).
\end{multline*}

Consider the third term in figure brackets in~(\ref{sphere-asymp}) (the fourth one is treated analogously).
The corresponding integral over $r>1/s$, up to a constant multiple, equals
\begin{multline*}
  2 \left(\frac{1}{s}\right)^{N/2-1} 
  \int_{1/s}^\infty r^{-N/2+1} e^{-2 i r s} e^{-i r p} A(-\theta, \omega, r) dr
  \\= \frac{i}{s^{N/2}}
  \int_{1/s}^\infty r^{-N/2+1} \partial_r(e^{-2 i r s}) e^{-i r p} A(-\theta, \omega, r) dr
  = \frac{-i e^{-2i} e^{-i p/s}\, A(-\theta, \omega, 1/s)}{s} 
  \\- \frac{i}{s^{N/2}} \int_{1/s}^\infty e^{-2 i r s} \partial_r\left(e^{-i r p} r^{-N/2+1} A(-\theta, \omega, r)\right) dr.
\end{multline*}
Now by invoking condition~(\ref{U-cond-pointwise}) with $\alpha=0$, $k\leqslant 1$, and sufficiently large $\ell$,
it is easy to show that the resulting expression equals $O(s^{-N/2+1-\varepsilon})$ as $s\to+\infty$. 

By applying~(\ref{U-cond-pointwise}) again, we may estimate the absolute value of the integral over $r<1/s$ in~(\ref{rSS}) by
\begin{equation*}
  C \int_0^{1/s} r^{N/2 - 2 + \varepsilon} dr
  = C s^{-N/2+1-\varepsilon}.
\end{equation*}

Thus we arrive at relation~(\ref{scat-data}) with
\begin{equation*}
  f(\theta,\omega,p) = \frac{e^{i\pi (n-d)/4}}{(2\pi)^{N/2+1}}
  \int_0^\infty r^{-N/2+1}
  \left\{e^{-i r p} A(\theta, \omega, r) + i^{d-n} e^{i r p} A(-\theta, -\omega, r)\right\} dr.
\end{equation*}
Now taking the continuation of the function $A$ to $r\in{\mathbb R}$ by the rule~(\ref{U-line})
leads to relation~(\ref{fUR}).

\section{Auxiliary assertions}\label{sec-aux}   
\begin{lemma}\label{lemma-Ceps}
Suppose $f(p)$ is a measurable function on ${\mathbb R}$ satisfying
\begin{equation*}
  |f(p)| \leqslant \frac{C}{(1+|p|)^{1+\varepsilon}}
\end{equation*}
for some $0<\varepsilon<1$.
Then its inverse Fourier transform $V(r)$ belongs to $C^\varepsilon({\mathbb R})$.
\end{lemma}
\begin{proof}
  The estimate for the norm $\|V\|_{L_\infty}$ is trivial.
  We shall estimate the difference $V(r)-V(r')$ for $|r-r'|\leqslant 1$, assuming, with no loss of generality, that $r'=0$.
  We have
\begin{multline*}
  2\pi |V(r) - V(0)| \leqslant \int_{\mathbb R} |f(p) (e^{i r p} - 1)| dp
  \\\lesssim \int_{|p| < 1/r} |f(p) (e^{i r p} - 1)| dp  + \int_{|p| > 1/r} |f(p)| dp
  \leqslant r \int_{|p| < 1/r} |f(p) p|\, dp  + \int_{|p| > 1/r} |f(p)|\, dp.
\end{multline*}
Both terms of the resulting expression are bounded by $C r^\varepsilon$.
\end{proof}

\begin{lemma}\label{lemma-V-eps}
  Suppose $f(p)$ is a Lipshitz function in ${\mathbb R}$ satisfying
\begin{equation*}
  f(\pm\infty) = 0,
  \quad
  |\partial f(p)| \leqslant \frac{C}{(1+|p|)^{1+\varepsilon}} 
\end{equation*}
for some $0<\varepsilon<1$.
Then its inverse Fourier transform $V(r)$ is a regular function on ${\mathbb R}$, continuous on ${\mathbb R}\setminus\{0\}$,
and we have the following estimate 
\begin{equation*}
  |V(r)| \leqslant \frac{C}{|r|^{1-\varepsilon}}.
\end{equation*}
\end{lemma}
\begin{proof}
  Applying Lemma~\ref{lemma-Ceps} to the function $\partial f(p)$ gives that the function $w(r) = r V(r)$ belongs to $C^\varepsilon({\mathbb R})$.
We now show that $w(0) = 0$.
Let
\begin{equation*}
  \chi\in C_0^\infty({\mathbb R}), \quad \int_{\mathbb R} \chi(r) dr = 1, \quad
  \chi_h(r) = h^{-1} \chi(h^{-1} r), \quad h > 0.
\end{equation*}
Continuity of the function $w(r)$ implies
\begin{equation*}
  w(0) = \lim_{h\to 0} \langle{}w, \chi_h\rangle. 
\end{equation*}
Since $\check \chi_h(p) = \check\chi(h p)$, we have
\begin{equation*}
  \langle{}w, \chi_h\rangle = \langle\hat w, \check \chi_h\rangle = i \langle\partial f, \check \chi_h\rangle = i \int_{\mathbb R} \partial f(p) \check\chi(h p) dp.
\end{equation*}
Taking into account that $\check\chi\in \mathcal{S}({\mathbb R})$, we may perform integrating by parts
in the last integral, which gives
\begin{equation}
  -h \int_{\mathbb R} f(p) (\partial\check\chi)(h p) dp = -\int_{\mathbb R} f(p/h) \partial\check\chi(p) dp.
  \label{f-checkchi}
\end{equation}
It is easy to infer from the assumptions of the lemma, by using Newton-Leibniz formula,
that
\begin{equation*}
  (1+|p|)^\varepsilon f \in L_\infty({\mathbb R}).
\end{equation*}
Therefore $|f(p/h)| \leqslant C (1+|p|/h)^{-\varepsilon} \to 0$ as $h\to 0$ whenever $p\ne 0$.
Besides, the integrand in the last integral is bounded by an integrable majorant of the form $C |\partial\check\chi(p)|$, so
the limit of the integral as $h\to 0$, together with $w(0)$, is equal to zero.

Thus $w(r)/r$ is a regular function on ${\mathbb R}$, continuous on ${\mathbb R}\setminus\{0\}$, and we have
\begin{equation*}
  \left|\frac{w(r)}{r}\right| \leqslant \frac{C}{|r|^{1-\varepsilon}}.
\end{equation*}
It remains to show that $V(r) = w(r)/r$.

The equality
\begin{equation*}
  r (V(r) - w(r)/r) = 0
\end{equation*}
(which is understood as the equality of distributions in ${\mathbb R}$)
implies that
\begin{equation*}
  V(r) - w(r)/r = c\, \delta(r)
\end{equation*}
for some constant $c$.
We will show that $c=0$.
It follows from this relation that, for $h>0$, we have
\begin{equation*}
  h \langle{}V, \chi_h\rangle = \int_{\mathbb R} w(r) r^{-1} \chi(r/h) dr + c \chi(0).
\end{equation*}
The integral in the resulting expression tends to zero as $h\to 0$.
On the other hand, we have
\begin{equation*}
  h \langle{}V, \chi_h\rangle = h \langle{}f, \check \chi_h\rangle = h \int_{\mathbb R} f(p) \check\chi(h p) dp = \int_{\mathbb R} f(p/h) \check\chi(p) dp.
\end{equation*}
We obtain the integral, similar to that in~(\ref{f-checkchi}). It also tends to zero as $h\to 0$.
Our argument shows that $c \chi(0) = 0$.
Now by choosing a function $\chi(r)$ that is nonzero in $r=0$, we prove the desired assertion.
\end{proof}

\begin{lemma}\label{lemma-V-eps-high-order}
Suppose $f(p)$ is a function from $C^K({\mathbb R})$, $K\geqslant 1$, such that
\begin{equation}
  f(\pm\infty) = 0,
  \quad
  |\partial^k f(p)| \leqslant \frac{C}{(1+|p|)^{k+\varepsilon}}, \quad 1\leqslant k \leqslant K,
  \label{df-est-K}
\end{equation}
for some $0<\varepsilon<1$.
Then its inverse Fourier transform $V(r)$ is a regular function on ${\mathbb R}$,
belonging to $C^{K-1}({\mathbb R}\setminus\{0\})$, and the following estimate holds
\begin{equation}
  |\partial^{k-1} V(r)| \leqslant \frac{C}{|r|^{k-\varepsilon}}, \quad 1\leqslant k\leqslant K.
  \label{V-est}
\end{equation}
\end{lemma}
\begin{proof}
The argument is based on the mathematical induction.
In the case $K=1$, Lemma~\ref{lemma-V-eps} applies.
Assume $K\geqslant 2$ and the statement holds with $K$ replaced by $K-1$.
Consider the function $\partial(p f(p))$. The latter satisfies conditions~(\ref{df-est-K}) with $K$ replaced by $K-1$.
Applying the induction hypothesis to this function gives that its inverse Fourier transform $W(r)$ is regular in ${\mathbb R}$ and satisfies
\begin{equation}
  W \in C^{K-2}({\mathbb R}\setminus\{0\}),
  \quad
  |\partial^{K-2} W(r)| \leqslant \frac{C}{|r|^{K-1-\varepsilon}}.
  \label{W-est}
\end{equation}
The function $W(r)$, up to a constant multiple, is equal to $r\partial V(r)$,
so the function $\partial^{K-2} W(r)$, up to a constant multiple, is equal to 
\begin{equation*}
  r \partial^{K-1}V(r) + (K-2) \partial^{K-2} V(r).
\end{equation*}
The second term in this expression can be estimated by applying the induction hypothesis to the function $f$
and taking $k=K-1$ in~(\ref{V-est}): 
\begin{equation*}
  V\in C^{K-2}({\mathbb R}\setminus\{0\}), \quad |\partial^{K-2} V(r)| \leqslant \frac{C}{|r|^{K-1-\varepsilon}}.
\end{equation*}
Together with~(\ref{W-est}) this leads to estimate~(\ref{V-est}) with $k=K$.
For $k<K$, this estimate is contained in the induction hypothesis.
\end{proof}

\begin{lemma}\label{lemma-V-eps-high-order-decay}
  Suppose $f(p)$ is a function from $C^\infty({\mathbb R})$ satisfying~(\ref{df-est-K}) for any $K$.
  Then its inverse Fourier transform $V(r)$ is a regular function in ${\mathbb R}$,
  belonging to $C^\infty({\mathbb R}\setminus\{0\})$, and the following estimate holds
\begin{equation}
  |\partial^k V(r)| \leqslant C_{\ell,k} |r|^{-\ell}, \quad \ell, k\geqslant 0,
  \label{V-est-super}
\end{equation}
\end{lemma}
\begin{proof}
Consider first the case $k=0$.  
Observe that the function $\partial^\ell f(p)$ satisfies conditions~(\ref{df-est-K}) for any $K$.
The inverse Fourier transform of this function, up to a constant multiple, is equal to $r^\ell V(r)$.
By applying Lemma~\ref{lemma-V-eps-high-order} to $\partial^\ell f(p)$ and setting $k=1$ in estimate~(\ref{V-est}),
we obtain that
\begin{equation*}
  |V(r)| \leqslant C_\ell |r|^{-\ell-1+\varepsilon}.
\end{equation*}
Since $\ell$ is arbitrary, we obtain~(\ref{V-est-super}) for $k=0$.

In the case of arbitrary $k$, we apply the mathematical induction.
Assume inequality~(\ref{V-est-super}) holds for $k<k_0$ and for any function that satisfies the conditions of the lemma.
Note that the function $\partial(p f(p))$ satisfies conditions~(\ref{df-est-K}) for any $K$.
By applying~(\ref{V-est-super}) to it with $k=k_0-1$, we obtain that
\begin{equation*}
  |\partial^{k_0-1} (r \partial V(r))| \leqslant C_{\ell,k_0} |r|^{-\ell}, \quad \ell\geqslant 0.
\end{equation*}
Since
\begin{equation*}
  \partial^{k_0-1} (r \partial V(r)) = r \partial^{k_0} V(r) + (k_0-1) \partial^{k_0-1} V(r),
\end{equation*}
and, by the induction hypothesis, the derivative $\partial^{k_0-1} V(r)$ is bounded by $C_{\ell,k_0} |r|^{-\ell}$, $\ell\geqslant 0$,
we obtain~(\ref{V-est-super}).
\end{proof}

\begin{lemma}\label{lemma-V-eps-high-order-final}
Suppose $f(p)$ is a function from $C^\infty({\mathbb R})$ satisfying~(\ref{df-est-K}) for any $K$.
Then its inverse Fourier transform $V(r)$ is a regular function in ${\mathbb R}$,
belonging to $C^\infty({\mathbb R}\setminus\{0\})$, and the following estimate is valid
\begin{equation}
  |\partial^k V(r)| \leqslant C_{\ell,k} r^{-k-1+\varepsilon} (1+|r|)^{-\ell}, \quad \ell, k\geqslant 0.
  \label{V-est-final}
\end{equation}
\end{lemma}
\begin{proof}
  The assertion follows from Lemmas~\ref{lemma-V-eps-high-order}, \ref{lemma-V-eps-high-order-decay}.
\end{proof}


\begin{thebibliography}{}

\bibitem{Blag-scatter}
  A.~S.~Blagoveshchensky, On Some New Well-Posed Problems for the Wave Equation,
  in {\em Proceedings of the V All-Union Symposium on Diffraction and Wave Propagation (1970)}, 29--35, Leningrad, Nauka  (1971) [in Russian].
  
\bibitem{Lax-Phillips}
  Lax, P.\,D., Phillips, R.\,S., {\em Scattering Theory}. Academic Press, New York--London (1967).

\bibitem{Moses}
  H. E. Moses, R. T. Prosser, Acoustic and Electromagnetic Bullets: Derivation of New
Exact Solutions of the Acoustic and Maxwell’s Equations, SIAM J. Appl. Math.,
50:5 (1990), 1325--1340.

\bibitem{Kis}
  A.~P.~Kiselev, Localized Light Waves: Paraxial and Exact Solutions of the Wave Equation (a Review),
  {\em Optics and Spectroscopy}, 102:4 (2007), 603--622.

\bibitem{Plachenov}
  A.~B.~Plachenov, Energy of Waves (Acoustic, Electromagnetic, Elastic) via Their Far-Field Asymptotics at Large Time,
  {\em Journal of Mathematical Sciences}, 277:4 (2023), 653--665.

\bibitem{BV}
  M. I. Belishev, A. F. Vakulenko, On a Control Problem for the Wave Equation in ${\mathbb R}^3$,
  {\em Journal of Mathematical Sciences}, 142 (2007), 2528--2539.
  
\bibitem{Blag-char}
A.~S.~Blagoveshchensky, On the Characteristic Problem for the Ultrahyperbolic Equation, {\em Matem. Sbornik},
63:105 (1964), No.~1, 137--168 [in Russian].


\bibitem{DMN}
  M.~N.~Demchenko, Asymptotic properties of solutions to a certain ultrahyperbolic equation,
  {\em Zap. Nauchn. Semin. POMI}, 516 (2022), 40--64
  [in Russian; English translation in https://arxiv.org/abs/2210.17413 ].


\bibitem{Hormander}
  L. H\"ormander, {\em The Analysis of Linear Partial Differential Operators I: Distribution Theory and Fourier Analysis},
  Classics in Mathematics, Springer-Verlag Berlin Heidelberg (2003).
  
\end{thebibliography}
\end{document}